\documentclass[12pt]{amsart}
\usepackage{amssymb}
\usepackage{fullpage}
\usepackage{url}
\begin{document}
\title{On Mullin's second sequence of primes}
\author{Andrew R.~Booker}
\thanks{The author was supported by an EPSRC fellowship}
\address{
Howard House\\
University of Bristol\\
Queens Ave\\
Bristol\\
BS8 1SN\\
United Kingdom
}
\email{\tt andrew.booker@bristol.ac.uk}
\date{}
\begin{abstract}
We consider the second of Mullin's sequences of prime numbers related to
Euclid's proof that there are infinitely many primes. We show in
particular that it omits infinitely many primes, confirming a conjecture
of Cox and van der Poorten.
\end{abstract}
\maketitle
\newtheorem{theorem}{Theorem}
\newtheorem{proposition}{Proposition}
\newtheorem{lemma}{Lemma}
\newcommand{\F}{\mathbb{F}}
\newcommand{\Z}{\mathbb{Z}}
\newcommand{\Q}{\mathbb{Q}}
\newcommand{\R}{\mathbb{R}}
\section{Introduction}
In \cite{mullin}, Mullin constructed two sequences of prime numbers
related to Euclid's proof that there are infinitely many primes.
For the first sequence, say $\{p_n\}_{n=1}^{\infty}$,
we take $p_1=2$ and define $p_{n+1}$ to be the
smallest prime factor of $1+p_1\cdots p_n$.  The second sequence,
$\{P_n\}_{n=1}^{\infty}$, is defined similarly, except that we replace
the words ``smallest prime factor'' by ``largest prime factor''.
These are sequences A000945 and A000946 in the OEIS
\cite{OEIS}, and the first few terms of each are shown below.
\begin{table}[h]
\caption{First ten terms of Mullin's sequences}
\begin{tabular}{rll}
$n$&$p_n$&$P_n$\\ \hline
$1$ & $2$ & $2$ \\
$2$ & $3$ & $3$ \\
$3$ & $7$ & $7$ \\
$4$ & $43$ & $43$ \\
$5$ & $13$ & $139$ \\
$6$ & $53$ & $50207$ \\
$7$ & $5$ & $340999$ \\
$8$ & $6221671$ & $2365347734339$ \\
$9$ & $38709183810571$ & $4680225641471129$ \\
$10$& $139$ & $1368845206580129$
\end{tabular}
\end{table}

Mullin then asked whether every prime is contained in each of these
sequences, and if not, whether they are recursive, i.e.\
whether there is an algorithm to decide if a given prime occurs
or not.\footnote{Mullin also asked whether the second sequence might be
monotonic (and hence recursive); this was answered negatively by Naur
\cite{naur}, who was the first to compute it beyond the 9th term.
However, it remains an open question whether there are infinitely
many $n$ such that $P_n>P_{n+1}$.}
Almost nothing related to this is known for the first sequence, though
Shanks \cite{shanks} conjectured on probabilistic grounds that
it contains every prime; we briefly discuss this conjecture and some
variants in Section 2 below.  Concerning the second sequence, Cox and
van der Poorten \cite{cv} showed that, apart from the first four terms
$2$, $3$, $7$ and $43$, it omits all the primes less than $53$; it is
straightforward to extend this to the remaining primes less than $79$
by applying their method using the most recent computations of $P_n$,
due of Wagstaff \cite{wagstaff}.  In response to Mullin's questions,
Cox and van der Poorten conjectured that
infinitely many primes are omitted, and that their
method would always work to decide whether a given prime occurs;
moreover, they showed that at least one of their conjectures is true.
The main point of this paper is to prove the first of these conjectures.
Precisely, we show the following.
\begin{theorem}
The sequence $\{P_n\}_{n=1}^{\infty}$ omits infinitely many primes.
If $\{Q_n\}_{n=1}^{\infty}$ denotes the sequence of omitted primes in
increasing order, then
$$
\limsup_{n\to\infty}\frac{\log Q_{n+1}}{\log(Q_1\cdots Q_n)}
\le\frac1{4\sqrt{e}-1}=0.1787\ldots.
$$
\end{theorem}
We note that although our method of proof allows us to bound each omitted
prime $Q_n$ in terms of the previous ones, it is not constructive;
in particular, Mullin's second question remains open (see Theorem 2
below, however).

The number $\frac1{4\sqrt{e}-1}$ in the theorem
is related to the best known bound
$O\Bigl(p^{\frac1{4\sqrt{e}}+o(1)}\Bigr)$ for the
least quadratic non-residue (mod $p$).  This was first shown by Burgess
\cite{burgess1}, based on an argument of Vinogradov; apart from refinements
of the $o(1)$, it has not been improved upon in over 50 years.
However, if the Generalized Riemann Hypothesis for quadratic Dirichlet
$L$-functions is true then one can show the much stronger bound
$Q_{n+1}=O\bigl(\log^2(Q_1\cdots Q_n)\bigr)$,
from which it follows that
$$
\#\{n:Q_n\le x\}\gg \frac{\sqrt{x}}{\log{x}}
$$
for large $x$.  Even this seems far from the truth; indeed, it
is likely that the set of primes that occur in $\{P_n\}_{n=1}^{\infty}$
has density $0$.
While we have not been able to prove that unconditionally, by refining
Cox and van der Poorten's argument on the relationship between their
conjectures, we can show the following.
\begin{theorem}
If $\{P_n\}_{n=1}^{\infty}$ is not recursive then it
has logarithmic density $0$ in the primes, i.e.\
$$
\lim_{x\to\infty}
\frac{\sum_{\substack{p\le x\\p\in\{P_1,P_2,\ldots\}}}\frac1p}
{\sum_{\substack{p\le x\\p\text{ \rm prime}}}\frac1p}
=0.
$$
\end{theorem}

\section{Variants}
Before embarking on the proofs of Theorems 1 and 2, we set our results
in context by comparing to a few variants
of the sequence $\{P_n\}_{n=1}^{\infty}$.
\begin{enumerate}
\item As mentioned above, very little is known about Mullin's first
sequence $\{p_n\}_{n=1}^{\infty}$.  Shanks reasoned that as $n$ increases,
the numbers $t_n=p_1\cdots p_n$ should vary randomly among the invertible
residues classes (mod $p$) for any fixed prime $p$, until $p$ occurs in
the sequence, after which point $t_n\equiv0\;(\text{mod }p)$.  If $p$
does not occur then this is violated, since $t_n$ is always invertible
(mod $p$) but falls into the residue class of $-1$ at most finitely
many times.  As no one has found any reason to suggest that $t_n$ does
not vary randomly (mod $p$), this is certainly compelling.  However, there
is reason to tread cautiously, first because Kurokawa and Satoh \cite{ks}
have shown that an analogue of this conjecture for the Euclidean domains
$\F_p[x]$ is false in general, and second because of what happens in
the next variant that we consider.
\item In the second variant, instead of just introducing one new prime
at each step, we add in all prime divisors of $1$ plus the product of
the previously constructed primes.  In symbols, we set $S_0=\emptyset$
and define $S_n$ recursively by
$$
S_{n+1}=S_n\cup
\left\{p:p\text{ prime and }
p\Bigl|\Bigl(1+\prod_{s\in S_n}s\Bigr)\right\}.
$$
This is related to Sylvester's sequence
$\{s_n\}_{n=1}^{\infty}$, defined by
$s_n=1+\prod_{i=0}^{n-1}s_i$, or equivalently,
$s_0=2$, $s_{n+1}=1+s_n(s_n-1)$.
More precisely, there is empirical evidence to suggest that $s_n$
is always squarefree, and if that is the case then
$$
\prod_{p\in S_n}p=\prod_{i=0}^{n-1}s_i.
$$
In particular, each prime that we construct this way divides some
Sylvester number.  One could try applying the same sort reasoning as in
Shanks' conjecture for this sequence, but it turns out that there is a
conspiracy preventing this from working, since $s_n$ can be described by
a one-step recurrence.  In fact, Odoni \cite{odoni} showed that the set
of primes dividing a Sylvester number has density $0$.  Thus, perhaps
counterintuitively, the greedy algorithm of adding in all prime divisors
likely yields a very thin subset of the primes.
\item Pomerance considered the following variant (unpublished, but see
\cite[\S1.1.3]{cp}).  Let $r_1=2$, and define $r_{n+1}$ recursively to be the
smallest prime number which is not one of $r_1,\ldots,r_n$
and divides a number of the form $d+1$, where $d|r_1\cdots r_n$.
This is in some sense even greedier than the previous variant, but the
fact that we can choose proper divisors $d$ of $r_1\cdots r_n$ prevents the
numbers from growing out of control.  Thus, Pomerance showed that every
prime does indeed occur in this sequence, and in fact $r_n$ is just the
$n$th prime number for $n\ge 5$.
\item Each variant has an analogue with the $+1$ in the definition
replaced by $-1$.  For instance, Selfridge (unpublished, but see
\cite{gn}) considered the sequence $\{\widetilde{P}_n\}_{n=1}^{\infty}$
where $\widetilde{P}_1=3$ and $\widetilde{P}_{n+1}$ is the largest prime
factor of $\widetilde{P}_1\cdots \widetilde{P}_n-1$.  He showed that it
omits some primes, analogous to the result of Cox and van der Poorten
for $\{P_n\}_{n=1}^{\infty}$.  Likewise, with some small modifications
to the proof, it is not hard to see that Theorem 1 remains true with
$\{P_n\}_{n=1}^{\infty}$ replaced by $\{\widetilde{P}_n\}_{n=1}^{\infty}$.
\end{enumerate}

\section{Proofs}
We begin by reviewing the method of \cite{cv}.
For a positive integer $n$, suppose that $1+P_1\cdots P_n$ has the
factorization
\begin{equation}
\tag{$\ast$}
\label{PQ}
1+P_1\cdots P_n=q_1^{k_1}\cdots q_r^{k_r},
\end{equation}
where $q_1<\ldots<q_r$ are prime and $q_r=P_{n+1}$.
Observe that the left-hand side
is $\equiv3\;(\text{mod }4)$, so that
$$
\left(\frac{-4}{q_1}\right)^{k_1}\cdots\left(\frac{-4}{q_r}\right)^{k_r}=-1,
$$
where $\bigl(\frac{a}{b}\bigr)$ denotes the Kronecker symbol.
Similarly, if $d$ is a fundamental discriminant dividing
$P_1\cdots P_n$ then the left-hand side is $\equiv1\;(\text{mod }d)$, so
that
$$
\left(\frac{d}{q_1}\right)^{k_1}\cdots\left(\frac{d}{q_r}\right)^{k_r}=1.
$$
Cox and van der Poorten considered values of $d$ for which $|d|$ is one 
of the known $P_i$, thus obtaining a system of equations which they
attempted to solve by linear algebra over $\F_2$.
As more of the $P_i$ become known, one adds more and more constraints
that must be satisfied by the small primes $q$ which have not yet
occurred, and one can hope eventually to reach an inconsistent system.
There is no known reason to believe that the equations for the various
$P_i$ are related, and this motivates their conjectures.

An equivalent formulation of their method is to
look for a fundamental discriminant $d$ composed of known
$P_i$ such that $\bigl(\frac{d}{q}\bigr)=\bigl(\frac{-4}{q}\bigr)$
for the first several primes $q$ which are not known to occur.
This is the approach that we will take, as outlined in the following
lemmas.
\begin{lemma}
Let $\chi\;(\text{mod }q)$ be a non-principal quadratic character, not
necessarily primitive.  Then there is a prime number
$n\ll_{\varepsilon} q^{\frac1{4\sqrt{e}}+\varepsilon}$
such that $\chi(n)=-1$.
\end{lemma}
\begin{proof}
Let $n$ be the smallest positive integer such that $\chi(n)=-1$.  It is
clear that $n$ must be prime, so it suffices to prove the upper bound.
This is essentially a special case of \cite[Theorem 1]{lw}, except for
the technical point that $q$ need not be cubefree.

To circumvent that, we
factor $\chi=\chi_0\chi_1$ where $\chi_0\;(\text{mod }q_0)$ is trivial
and $\chi_1\;(\text{mod }q_1)$ is a primitive quadratic character.
Note that if we replace $q_0$ by
$q_0'=\prod_{\substack{p|q_0\\p\nmid q_1}}p$
and $\chi_0$ by the trivial character $\chi_0'\;(\text{mod }q_0')$,
then $\chi'=\chi_0'\chi_1$ satisfies $\chi'(m)=\chi(m)$ for every $m$.
Thus, we may assume without loss of generality that $q_0$ is squarefree
and $(q_0,q_1)=1$.

Moreover, $\pm q_1$ is a fundamental discriminant, so in fact
$q=q_0q_1$ is cubefree except possibly for a factor of $8$.
Even if $8|q$, one can see that Burgess' bounds \cite[Theorem 2]{burgess2},
on which \cite[Theorem 1]{lw} is based, continue to hold at the
expense of a worse implied constant. (See \cite[(12.56)]{ik}
for a precise statement of this type.)  The result follows.
\end{proof}

\begin{lemma}
Let $q_1,\ldots,q_r$ be pairwise relatively prime positive integers.
For each $i=1,\ldots,r$, let $\chi_i\;(\text{mod }q_i)$ be a non-principal 
quadratic character, not necessarily primitive, and let
$\epsilon_i\in\{\pm1\}$.
Then there is a squarefree positive integer $n$ with at most $r$ prime
factors, each
$\ll_{\varepsilon} (q_1\cdots q_r)^{\frac1{4\sqrt{e}}+\varepsilon}$,
such that $\chi_i(n)=\epsilon_i$ for all $i=1,\ldots,r$.
\end{lemma}
\begin{proof}
Let $\psi_i$ be the principal character (mod $q_i$) for $i=1,\ldots,r$,
and set $q=q_1\cdots q_r$.
For each non-empty subset $S\subset\{1,\ldots,r\}$ we define a character
$\chi_S\;(\text{mod }q)$ by
$$
\chi_S(n)=\prod_{i=1}^r
\begin{cases}
\chi_i(n)&\text{if }i\in S,\\
\psi_i(n)&\text{if }i\notin S.
\end{cases}
$$
Note that $\chi_S$ must be non-trivial since the $q_i$ are pairwise
relatively prime.  By Lemma 1, there is a prime
$n_S\ll_{\varepsilon} q^{\frac1{4\sqrt{e}}+\varepsilon}$
such that $\chi_S(n_S)=-1$.
Further, we associate to $S$ two vectors in $\F_2^r$.  The first is the
characteristic vector
$v_S=(a_1,\ldots,a_r)$, defined by
$$
a_i=
\begin{cases}
1&\text{if }i\in S,\\
0&\text{if }i\notin S.
\end{cases}
$$
The second is the unique vector
$w_S=(b_1,\ldots,b_r)$ such that
$\chi_i(n_S)=(-1)^{b_i}$ for $i=1,\ldots,r$.
These vectors have scalar product $v_S\cdot w_S=1$ since
$\chi_S(n_S)=-1$.

We claim that
$\bigl\{w_S:\emptyset\ne S\subset\{1,\ldots,r\}\bigr\}$ spans
$\F_2^r$.  If not then there would be a non-zero linear functional which
vanishes at each such $w_S$, i.e.\ a non-zero $v\in\F_2^r$ with
$v\cdot w_S=0$ for all $S\ne\emptyset$.
However, this is impossible since the $v_S$ exhaust
all non-zero vectors in $\F_2^r$.

Therefore, there is a set $T$ of non-empty subsets of $\{1,\ldots,r\}$
such that $\{w_S:S\in T\}$ is a basis for $\F_2^r$.  It follows that
the numbers $n_S$ for $S\in T$ are distinct primes, and as $n$ ranges
over the divisors of $\prod_{S\in T}n_S$,
$(\chi_1(n),\ldots,\chi_r(n))$ ranges over all elements of
$\{\pm1\}^r$.
\end{proof}

\subsubsection*{Proof of Theorem 1}
Let $Q_1,\ldots,Q_r$ be the first $r$ omitted primes.  (We allow $r=0$
to start the argument, with the understanding that $Q_1\cdots Q_r=1$ in
that case.)  Suppose that all other primes up to some number $x\ge 3$
eventually occur, and let $p=P_{n+1}\le x$ be the last to occur.
Then except for $Q_1,\ldots,Q_r$, all primes below $p$ must occur before $p$,
so \eqref{PQ} takes the form
$$
1+P_1\cdots P_n=Q_1^{k_1}\cdots Q_r^{k_r}\cdot p^k
$$
for some $k,k_1,\ldots,k_r\in\Z_{\ge0}$.
Now, applying Lemma 2 with the characters
$$
\left(\frac{-4}{\cdot}\right),
\left(\frac{\cdot}{p}\right)\text{ and }
\left(\frac{\cdot}{Q_1}\right),\ldots,\left(\frac{\cdot}{Q_r}\right),
$$
we can find a squarefree positive integer $d\equiv1\;(\text{mod }4)$
such that
$$
\left(\frac{d}{p}\right)=\left(\frac{-4}{p}\right),
\left(\frac{d}{Q_i}\right)=\left(\frac{-4}{Q_i}\right)
\text{ for }i=1,\ldots,r,
$$
and with all prime factors of $d$ bounded by
$O_{\varepsilon}\!\left((pQ_1\cdots Q_r)^{\frac1{4\sqrt{e}}+\varepsilon}\right)$.
Since $p\le x$ and $\frac1{4\sqrt{e}}<1$,
this bound must fall below $x$ for large
enough $x$, and in fact it is not hard to see that there is such an
$x\ll_{\varepsilon}(Q_1\cdots Q_r)^{\frac1{4\sqrt{e}-1}+\varepsilon}$.
This is a contradiction, and thus
there must be another omitted prime
$Q_{r+1}\ll_{\varepsilon}(Q_1\cdots Q_r)^{\frac1{4\sqrt{e}-1}+\varepsilon}$.
\qed

\medskip
The proof of Theorem 2 is based on the following generalization of the
method of Cox and van der Poorten.
For each $i=1,2,\ldots$, let $g_i$ be the smallest positive primitive root
(mod $P_i^2$), and let
$l_i:(\Z/P_i^2\Z)^{\times}\to\Z/P_i(P_i-1)\Z$ be the base-$g_i$
logarithm.  Suppose that we have computed $P_1,\ldots,P_N$.
Note that if $n\ge N$ then for
any $i\le N$, the left-hand side of \eqref{PQ}
is $\equiv1\;(\text{mod }P_i)$ but
$\not\equiv1\;(\text{mod }P_i^2)$ since the $P$'s are distinct.
Thus,
$k_1l_i(q_1)+\ldots+k_rl_i(q_r)\equiv0\;(\text{mod }P_i-1)$, but is non-zero
(mod $P_i$).  In other words, there is a vector
$b_i\in\F_{P_i}^r$ such that $b_i\cdot(k_1,\ldots,k_r)\ne0\in\F_{P_i}$.
On the other hand, we can construct other constraints (mod $P_i$) by
considering \eqref{PQ} modulo any $P_j$ for which
$P_j\equiv1\;(\text{mod }P_i)$ (if there are any).
If $P_j$ is such a prime then
$k_1l_j(q_1)+\ldots+k_rl_j(q_r)\equiv0\;(\text{mod }P_i)$,
i.e.\ there is a vector
$v_{ij}\in\F_{P_i}^r$ such that $v_{ij}\cdot(k_1,\ldots,k_r)=0\in\F_{P_i}$.

Thus, we can try to prove that $q_r$ is omitted by finding a linear
combination of the $v_{ij}$ which yields $b_i$.  For $i=1$, this
is equivalent to Cox and van der Poorten's method.  If that fails
to exclude $q_r$ then we can try $i=2$, and so on.  Note that from a
practical standpoint, one will accumulate equations modulo $P_1=2$ far
more quickly than for the other primes.  Thus, the greatest chance of
success is with $i=1$, so this is unlikely to yield any improvement over
their method in practice.  However, as our proof will show, the other
primes become useful if there is a conspiracy which makes their method fail.

\begin{lemma}
Let $n$ be a squarefree positive integer, $q$ an integer which is
relatively prime to $n$ and not
a perfect $p$th power for any prime $p|n$, and $d$ a divisor of
$n$.  Then the field $L=\Q(\sqrt[d]{q},e^{2\pi i/n})$ is normal over
$\Q$ and has degree $[L:\Q]=d\varphi(n)$.  Further, a rational prime $p$
not dividing the discriminant of $L$ splits completely in $L$ if and
only if $p\equiv1\;(\text{mod }n)$ and
$\exists x\in\Z$ such that $x^d\equiv q\;(\text{mod }p)$.
\end{lemma}
\begin{proof}[Proof (adapted from \cite{li}, Lemmas 3.1 and 3.2)]
First note that $L$ is the splitting field of $(x^d-q)(x^n-1)$, so it is
normal over $\Q$.  Set $\zeta_n=e^{2\pi i/n}$, and let $K=\Q(\zeta_n)$
be the corresponding cyclotomic field.  Then $K$ has degree $\varphi(n)$
over $\Q$, so to establish the formula for $[L:\Q]=[L:K][K:\Q]$,
it suffices to show that $x^d-q$ is irreducible over $K$.

To that end, we first show that $\sqrt[p]{q}\notin K$ for any prime
divisor $p|d$.  If $p$ is odd then $\Q(\sqrt[p]{q})\subset\R$ is not
normal over $\Q$ since it has non-real conjugates.  On the other hand,
every subfield of $K$ is normal over $\Q$ since $K$ is an abelian
extension, and thus $\sqrt[p]{q}\notin K$.  This argument fails if
$p=2$, but in that case it follows from class field theory that the
quadratic subfields of $K$ are exactly those of the form $\Q(\sqrt{D})$
for fundamental discriminants $D|n$.  Since $(q,n)=1$, $\Q(\sqrt{q})$
is not among them, so the claim still holds.

Next, suppose that $f\in K[x]$ is a monic irreducible factor of $x^d-q$, of
degree $d'<d$.  Note that over $L$ we have the factorization
$$
x^d-q=\prod_{j=1}^d\bigl(x-\zeta_d^j\sqrt[d]{q}\bigr),
$$
where $\zeta_d=\zeta_n^{n/d}$ is a primitive $d$th root of unity. Thus,
the constant term of $f$ must take the form $(-1)^{d'}\zeta_n^k q^{d'/d}$
for some integer $k$.  Hence $q^{d'/d}\in K$, and by the Euclidean algorithm
we can improve this to $q^{(d',d)/d}\in K$.  However, since $0\ne d'<d$,
there is a prime $p\bigl|\frac{d}{(d',d)}$.  This implies that
$\sqrt[p]{q}\in K$, in contradiction to the above, and thus $x^d-q$ is
irreducible over $K$, as claimed.

For the final statement, it is well-known that a rational prime $p$
splits completely in $K=\Q(\zeta_n)$ if and only if
$p\equiv1\;(\text{mod }n)$, and this is a necessary condition
for $p$ to split completely
in $L\supset K$.  If $p\equiv1\;(\text{mod }n)$, let $\mathfrak{p}$
be any of the $\varphi(n)$ primes of $K$ dividing $p\mathfrak{o}_K$,
where $\mathfrak{o}_K$ is the ring of integers of $K$.  If $p$ does not
divide the discriminant of $L$ then $\mathfrak{p}$ splits completely
in $L$ if and only if $x^d-q$ has $d$ roots in the residue field
$\mathfrak{o}_K/\mathfrak{p}\cong\F_p$, which in turn happens if and
only if $q$ has a $d$th root (mod $p$).
\end{proof}

\begin{lemma}
Let $m$ be a squarefree positive integer and $q$ an integer which is
relatively prime to $m$ and not
a perfect $p$th power for any prime $p|m$.  Then the set of primes
$p$ for which $x^m\equiv q\;(\text{mod }p)$ is solvable
has natural density $\frac{\varphi(m)}{m}$.
\end{lemma}
\begin{proof}
Note that the number of solutions of
$x^m\equiv q\;(\text{mod }p)$
is the same as that of
$x^{(m,p-1)}\equiv q\;(\text{mod }p)$.
For large $y>0$, we thus want to estimate the fraction
$$
\begin{aligned}
\frac1{\pi(y)}&\sum_{p\le y}
\begin{cases}
1&\text{if }x^{(m,p-1)}\equiv q\;(\text{mod }p)\text{ is solvable},\\
0&\text{otherwise}
\end{cases}
\\
&=\sum_{d|m}\frac1{\pi(y)}\sum_{\substack{p\le y\\(m,p-1)=d}}
\begin{cases}
1&\text{if }x^d\equiv q\;(\text{mod }p)\text{ is solvable},\\
0&\text{otherwise}
\end{cases}\\
&=\sum_{d|m}\sum_{e|\frac{m}{d}}\mu(e)
\frac1{\pi(y)}\sum_{\substack{p\le y\\p\equiv1\;(\text{mod }de)}}
\begin{cases}
1&\text{if }x^d\equiv q\;(\text{mod }p)\text{ is solvable},\\
0&\text{otherwise}
\end{cases}\\
&=\sum_{n|m}\sum_{d|n}\mu\!\left(\frac{n}{d}\right)
\frac1{\pi(y)}\sum_{\substack{p\le y\\p\equiv1\;(\text{mod }n)}}
\begin{cases}
1&\text{if }x^d\equiv q\;(\text{mod }p)\text{ is solvable},\\
0&\text{otherwise}.
\end{cases}
\end{aligned}
$$
By Lemma 3 and the Chebotarev Density Theorem, the inner sum over $p$
divided by $\pi(y)$ tends to $\frac1{d\varphi(n)}$ as $y\to\infty$.
(Note that the earlier Kronecker-Frobenius Density Theorem would be
enough here if we instead considered the logarithmic density.)
Thus, the set we are interested in has density
$$
\sum_{n|m}\sum_{d|n}\frac{\mu(n/d)}{d\varphi(n)}
=\sum_{n|m}\frac{\mu(n)}{\varphi(n)}\sum_{d|n}\frac{\mu(d)}{d}
=\sum_{n|m}\frac{\mu(n)}{n}=\frac{\varphi(m)}{m}.
$$
\end{proof}

\subsubsection*{Proof of Theorem 2}
Since $\{P_j\}_{j=1}^{\infty}$ is recursively enumerable, the only way
that it can fail to be recursive is if there is some $Q_r$
for which there is no algorithm to prove that it does not occur among
the $P_j$.  In particular, the general strategy described above must fail
to exclude $Q_r$, no matter how large we take $N$.

Note that for large
enough $N$, \eqref{PQ} will take the form
$$
1+P_1\cdots P_n=Q_1^{k_1}\cdots Q_r^{k_r}
$$
for $n\ge N$.  For $i=1,\ldots,N$, let $b_i,v_{ij}\in\F_{P_i}^r$ be
as described above.  Although we have restricted to $i\le N$,
we are free to consider arbitrarily large values of $j$ in this construction
by taking $n\ge j$ in \eqref{PQ}, so for each $i$ there
are potentially infinitely many suitable $j$.  In order to avoid eventually
concluding that $Q_r$ is omitted, $b_i$ must not be a linear combination
of the $v_{ij}$; in particular, the $v_{ij}$ span a proper subspace of
$\F_{P_i}^r$, so there is a non-zero vector $w_i\in\F_{P_i}^r$ such that
$v_{ij}\cdot w_i=0$ for every $j$ such that $P_j\equiv1\;(\text{mod }P_i)$.
By the Chinese Remainder Theorem, there are non-negative integers
$a_1,\ldots,a_r<P_1\cdots P_N$ such that
$(a_1,\ldots,a_r)\equiv w_i\;(\text{mod }P_i)$ for $i=1,\ldots,N$.
Set $q=Q_1^{a_1}\cdots Q_r^{a_r}$.  Then by construction, $q$ is not a perfect
$P_i$th power for any $i\le N$, but it is a $P_i$th power residue (mod
$P_j$) for all $j$ such that $P_j\equiv1\;(\text{mod }P_i)$. Note also
that $q$ is automatically a $P_i$th power residue (mod $P_j$) if
$P_j\not\equiv1\;(\text{mod }P_i)$.

It follows that the entire sequence $\{P_j:j=1,2,\ldots\}$ is a subset of the
primes modulo which $q$ is an $m$th power residue,
where $m=P_1\cdots P_N$.  By Lemma 4, that set has density
$$\frac{\varphi(m)}{m}=\prod_{i=1}^N\left(1-\frac1{P_i}\right).$$
Taking $N$ arbitrarily large, we have
$$
\limsup_{x\to\infty}\frac{\#\{j:P_j\le x\}}{\pi(x)}
\le\prod_{i=1}^{\infty}\left(1-\frac1{P_i}\right),
$$
with the understanding that the right-hand side is $0$ if the product
diverges.  In that case, $\{P_j\}_{j=1}^{\infty}$ has natural density
$0$, which in turn implies that the logarithmic density is $0$.  On the
other hand, if the product converges then so does
$\sum_{i=1}^{\infty}\frac1{P_i}$, which also implies that the
logarithmic density is $0$.

Finally, we remark that while it does not
necessarily follow that $\{P_j\}_{j=1}^{\infty}$ has a natural density,
the last inequality shows that
its upper density is strictly less than $1$; in fact, using just the
values in Table 1, we see that the upper density is at most $0.277056$.
\qed
\bibliographystyle{amsplain}

\begin{thebibliography}{10}

\bibitem{burgess1}
D.~A. Burgess, \emph{The distribution of quadratic residues and non-residues},
  Mathematika \textbf{4} (1957), 106--112. \MR{0093504 (20 \#28)}

\bibitem{burgess2}
\bysame, \emph{On character sums and {$L$}-series. {II}}, Proc. London Math.
  Soc. (3) \textbf{13} (1963), 524--536. \MR{0148626 (26 \#6133)}

\bibitem{cv}
C.~D. Cox and A.~J. Van~der Poorten, \emph{On a sequence of prime numbers}, J.
  Austral. Math. Soc. \textbf{8} (1968), 571--574. \MR{0228417 (37 \#3998)}

\bibitem{cp}
Richard Crandall and Carl Pomerance, \emph{Prime numbers: A computational
  perspective}, second ed., Springer, New York, 2005. \MR{2156291
  (2006a:11005)}

\bibitem{gn}
Richard Guy and Richard Nowakowski, \emph{Discovering primes with {E}uclid},
  Delta (Waukesha) \textbf{5} (1975), no.~2, 49--63. \MR{0384675 (52 \#5548)}

\bibitem{ik}
Henryk Iwaniec and Emmanuel Kowalski, \emph{Analytic number theory}, American
  Mathematical Society Colloquium Publications, vol.~53, American Mathematical
  Society, Providence, RI, 2004. \MR{2061214 (2005h:11005)}

\bibitem{ks}
Nobushige Kurokawa and Takakazu Satoh, \emph{Euclid prime sequences over unique
  factorization domains}, Experiment. Math. \textbf{17} (2008), no.~2,
  145--152. \MR{2433881 (2009k:11200)}

\bibitem{lw}
Yuk-Kam Lau and Jie Wu, \emph{On the least quadratic non-residue}, Int. J.
  Number Theory \textbf{4} (2008), no.~3, 423--435. \MR{2424331 (2009e:11191)}

\bibitem{li}
Shuguang Li, \emph{On extending {A}rtin's conjecture to composite moduli},
  Mathematika \textbf{46} (1999), no.~2, 373--390. \MR{1832628 (2002d:11118)}

\bibitem{mullin}
Albert~A. Mullin, \emph{Recursive function theory}, Bull. Amer. Math. Soc.
  \textbf{69} (1963), 737.

\bibitem{naur}
Thorkil Naur, \emph{Mullin's sequence of primes is not monotonic}, Proc. Amer.
  Math. Soc. \textbf{90} (1984), no.~1, 43--44. \MR{722412 (85i:11011)}

\bibitem{odoni}
R.~W.~K. Odoni, \emph{On the prime divisors of the sequence
  {$w_{n+1}=1+w_1\cdots w_n$}}, J. London Math. Soc. (2) \textbf{32} (1985),
  no.~1, 1--11. \MR{813379 (87b:11094)}

\bibitem{OEIS}
\emph{The {O}n-{L}ine {E}ncyclopedia of {I}nteger {S}equences}, Published
  electronically at \url{http://oeis.org}, 2011.

\bibitem{shanks}
Daniel Shanks, \emph{Euclid's primes}, Bull. Inst. Combin. Appl. \textbf{1}
  (1991), 33--36. \MR{1103634 (92f:11013)}

\bibitem{wagstaff}
Samuel~S. Wagstaff, Jr., \emph{Computing {E}uclid's primes}, Bull. Inst.
  Combin. Appl. \textbf{8} (1993), 23--32. \MR{1217356 (94e:11139)}

\end{thebibliography}
\providecommand{\bysame}{\leavevmode\hbox to3em{\hrulefill}\thinspace}
\providecommand{\MR}{\relax\ifhmode\unskip\space\fi MR }
\providecommand{\MRhref}[2]{%
  \href{http://www.ams.org/mathscinet-getitem?mr=#1}{#2}
}
\providecommand{\href}[2]{#2}

\end{document}